\documentclass[conference]{IEEEtran}
\usepackage[utf8]{inputenc}
\usepackage[T1]{fontenc}

\title{Nonlinear Approximation\\with Subsampled Rank-1 Lattices}
\author{\IEEEauthorblockN{Felix~Bartel}
\IEEEauthorblockA{Chemnitz University of Technology\\
felix.bartel@math.tu-chemnitz.de}
\and
\IEEEauthorblockN{Fabian~Taubert}
\IEEEauthorblockA{Chemnitz University of Technology\\
fabian.taubert@math.tu-chemnitz.de}
}
\date{\today}

\usepackage{amsmath, amssymb, mathtools, amsthm}
\usepackage{bm} \usepackage{dsfont} 

\newtheorem{theorem}{Theorem}[section]

\newtheorem{lemma}[theorem]{Lemma}

\newtheorem{example}[theorem]{Example}

\DeclareMathOperator*{\spn}{span}
\DeclareMathOperator*{\argmin}{arg\,min}
\DeclareMathOperator*{\sinc}{sinc}
\newcommand{\inv}{^{-1}}

\usepackage{color}
\definecolor{azure}{RGB}{0, 127, 255}
\definecolor{violet}{RGB}{161, 11, 112}
\definecolor{orange}{RGB}{243, 119, 53}

\begin{document} 

\maketitle

\begin{abstract}
    In this paper we approximate high-dimensional functions $f\colon\mathds T^d\to\mathds C$ by sparse trigonometric polynomials based on function evaluations.
    Recently it was shown that a dimension-incremental sparse Fourier transform (SFT) approach does not require the signal to be exactly sparse and is applicable in this setting.
    We combine this approach with subsampling techniques for \mbox{rank-1} lattices.
    This way our approach benefits from the underlying structure in the sampling points making fast Fourier algorithms applicable whilst achieving the good sampling complexity of random points (logarithmic oversampling).
    
    In our analysis we show detection guarantees of the frequencies corresponding to the Fourier coefficients of largest magnitude.
    In numerical experiments we make a comparison to full \mbox{rank-1} lattices and uniformly random points to confirm our findings.
\end{abstract}

\section{Introduction} 

The recovery of sparse signals or compressed sensing is a thoroughly studied problem in signal processing.
While many one-dimensional approaches exist \cite{GIIS14}, we have a look into the multivariate problem on the $d$-dimensional torus $\mathds T^d = (\mathds R_{/\mathds Z})^d$.
Given an $s$-sparse signal $f = \sum_{\bm k\in I} \hat f_{\bm k} \exp(2\pi\mathrm i\langle\bm k,\cdot\rangle)$, $|I| = s$, the problem is to recover $I\subset\mathds Z^d$ from function evaluations of the function $f$.
Here, a sparse Fourier transform (SFT) approach can be generalized to work in higher dimensions, cf.\ \cite{KPV21}.
However, when the signal is not exactly sparse and we approximate by $g = \sum_{\bm k\in I} \hat g_{\bm k} \exp(2\pi\mathrm i\langle\bm k,\cdot\rangle)$ for some $I\subset\mathds Z^d$, we obtain an additional error
\begin{equation*}
    \|f-g\|_{L_2}^2
    = \|f-P_{I} f\|_{L_2}^2 + \|P_{I} f - g\|_{L_2}^2 \,.
\end{equation*}
In this setting we have to find
\begin{itemize}
\item
    a suitable frequency set $I\subset\mathds Z^d$ to bound the first summand and
\item
    $\hat g_{\bm k} \in\mathds C$ approximating the true Fourier coefficients $\hat f_{\bm k} = \langle f, \exp(2\pi\mathrm i\langle\bm k, \cdot\rangle)\rangle_{L_2} \in\mathds C$ in order to bound the second summand.
\end{itemize}
Given a frequency set $I$, i.e., a suitable linear approximation space, the corresponding Fourier coefficients can be computed via least squares where error bounds are known, cf.\ \cite{Bar22}.
Thus, the main task is the detection of a frequency set $I$, which should optimally be the support of the Fourier coefficients $\hat f_{\bm k}$ with largest magnitude like in the best $m$-term approximation, cf.\ \cite[Section~1.7]{Tem11}.

Recent approaches include \cite{GI23} or \cite{KPT22} where arbitrary bounded orthonormal product basis were considered.
As an application \mbox{rank-1} lattices were used for the sampling similar to \cite{KPV21} in order to make use of fast Fourier algorithms.
This was then compared to random points, which have a better sampling complexity but lack fast algorithms.

In this paper we modify the approach from \cite{KPT22} to work with subsampled \mbox{rank-1} lattices utilizing recent subsampling techniques from \cite{BKPU22} combining the good sampling complexity with the fast algorithms.

The SFT techniques from \cite{KPT22} work for other bounded orthonormal product basis as well and the subsampling methods from \cite{BKPU22} for arbitrary $L_2$-Marcinkiewicz-Zygmund inequalities.
Therefore, the presented theory can be generalized but for the sake of readability we restrict ourselves to the torus $\mathds T^d$ and rank-1 lattices.

We will recap the ideas of an SFT approach in Section~\ref{sec:sft} followed by the subsampling techniques for rank-1 lattices in Section~\ref{sec:srank1}, where we will give an $L_2$-error bound for least squares approximation.
In Section~\ref{sec:sfftsub} we will combine the SFT approach with the subsampled \mbox{rank-1} lattices and show detection guarantees for the Fourier coefficients of largest magnitude in Theorem~\ref{thm:overall}.
Finally, we conclude with a numerical experiment in Section~\ref{sec:numerics} comparing \mbox{rank-1} lattices and random points with the subsampled \mbox{rank-1} lattices with respect to sampling complexity and runtime.
The proofs can be found in the supplementary material. \section{Prerequisites}

\subsection{Sparse Fourier Transform}\label{sec:sft} 

We briefly recall the key idea of a sparse Fourier transform (SFT) approach.
For a more detailed explanation see \cite{PV16}, \cite{KPV21}, \cite{KKV22}, or \cite{KPT22} for a more general version.
As stated in the introduction, the goal is to find frequencies $I\subset\mathds Z^d$ such that the target function $f\colon\mathds T^d\to\mathds C$ can be approximated well from $\spn\{\exp(2\pi\mathrm i\langle\bm k,\cdot\rangle)\}_{\bm k\in I}$.
In order to do so, we choose a suitable search space $\Gamma \subset \mathds Z^d$ and proceed in a dimension-incremental way:

\textbf{One-dimensional frequencies.}
We use the projections of $\Gamma$ to its $t$-th component $\mathcal P_{\{t\}}(\Gamma) \coloneqq \{k_t:\bm k\in\Gamma\}$, $t=1,\ldots,d$ for the candidate sets.
From these we construct frequency sets $I_{\{t\}} \subset\mathcal P_{\{t\}}(\Gamma)$, $t=1,\ldots,d$, consisting of the ``most important'', one-dimensional frequency components in the respective dimensions.

\textbf{Dimension-incremental step.}
We construct the next frequency set $I_{\{1,\dots,t+1\}}$ as a subset of the candidate set $(I_{\{1,\ldots,t-1\}} \times I_{\{t\}}) \cap \mathcal P_{\{1,\ldots,t\}}(\Gamma)$ consisting of the ``most important'', $t$-dimensional frequency components.

The output is the final frequency set $I \coloneqq I_{\{1,\dots,d\}}$ and it is left to refine the formulation of ``most important''.

Let $I^\star$ be the set of frequencies corresponding to the $s$ Fourier coefficients $\hat f_{\bm k}$ of largest magnitude for some sparsity $s\in\mathds N$.
Ideally, in the step $t-1\to t$ we want to find the frequencies $\mathcal P_{\{1,\dots, t\}}(I^\star)$.
The idea is to use an approximation of so-called projected Fourier coefficients
\begin{equation}\label{eq:proj_fc}
    \hat f_{\{1,\ldots,t\},\bm k}(\bm\xi)
    = \int_{\mathds T^{t}} f(\bm x, \bm\xi)\exp(-2\pi\mathrm i\langle \bm k, \bm x\rangle) \;\mathrm d\bm x \,,
\end{equation}
where $\bm x = (x_1, \dots, x_t) \in\mathds T^t$, $\bm\xi = (\xi_1, \dots, \xi_{d-t}) \in\mathds T^{d-t}$, and $f(\bm x, \bm\xi) = f(x_1, \dots, x_t, \xi_1, \dots, \xi_{d-t})$.
The name is based on the fact, that those values may be seen as the Fourier coefficients of the function $f(\cdot,\bm\xi) \in L_2(\mathds T^{t})$ with a fixed anchor $\bm\xi \in \mathds T^{d-t}$.
By the orthonormality of the Fourier basis we have
\begin{equation*}
    \hat f_{\{1,\ldots,t\},\bm k}(\bm\xi)
    = \sum_{\bm l\in\mathds Z^{d-t}}\hat f_{(\bm k,\bm l)} \exp(2\pi\mathrm i\langle\bm l,\bm\xi\rangle) \,,
\end{equation*}
i.e., the projected Fourier coefficient $\hat f_{\{1,\ldots,t\},\bm k}(\bm\xi)$ contains information on the Fourier coefficients with $\bm k\in\mathds Z^t$ in the first $t$ components of their frequencies.

The frequency $\bm k$ is likely to be important and should be included in $I_{\{1,\dots,t\}}$, if the absolute value $\vert \hat f_{\{1,\ldots,t\},\bm k}(\bm\xi)\vert$ is larger than some detection threshold $\delta'$.
In the algorithm, we carry out $r$ detection iterations with different, randomly drawn anchors $\bm\xi^i$, $i=1,\ldots,r$, to avoid cases where the factors $\exp(2\pi\mathrm i\langle\bm l,\bm\xi\rangle)$ cause an annihilation (which results in small projected Fourier coefficients, even though the corresponding frequency components $\bm k$ are important). The detection of the most important one-dimensional components $k_t$ in the first step of the SFT approach works analogously.

Finally, we need to discuss the approximation of \eqref{eq:proj_fc}.
A favorable method $\mathcal A$ should combine the following properties:
\begin{itemize}
    \item have small sample complexity;
    \item computationally fast, that is, both the construction of the sampling points $\bm\xi$ and the evaluation of the projected Fourier coefficients $\hat f_{\{1,\ldots,t\},\bm k}$ using the samples $f(\bm x,\bm\xi)$ can be performed efficiently.
    \item small error, such that the relative magnitude of the projected Fourier coefficients stays unharmed.
\end{itemize}
Note, that $\mathcal A$ has to be performed several times throughout the SFT approach in different dimensions up to $d$.
It is favorable to use different methods in the one- and multivariate steps using advantages of the respective methods.

\subsection{Subsampling of \mbox{rank-1} lattices}\label{sec:srank1} 

\mbox{Rank-1} lattices $\bm X_M = \{\bm x^1, \dots, \bm x^M\} \subset \mathds T^d$ consist of equispaced points on a line which wraps around the $d$-dimensional torus $\mathds T^d$, more precisely, for a generating vector $\bm z\in\mathds R^d$ and a lattice size $M\in\mathds M$ they are defined via
\begin{equation*}
    \bm X_M
    \coloneqq \Big\{ \frac{1}{M}(i\bm z \!\!\!\!\!\mod M\mathds 1)\in\mathds T^d : i = 0, \dots, M-1 \Big\} \,,
\end{equation*}
where the modulus operation is used entry-wise.
We will use them in the least squares approximation
\begin{equation*}
    S_{\bm X_M} f
    = \argmin_{g\in V} \sum_{i=1}^{M} | g(\bm x^i) - f(\bm x^i) |^2 \,,
\end{equation*}
where $V = \spn\{\exp(2\pi\mathrm i\langle\bm k, \cdot\rangle)\}_{k\in I}$.
By simple calculus we have for the Fourier coefficients $\bm{\hat g} = (\hat g_{\bm k})_{\bm k\in I}$ of $S_{\bm X_M}f = \sum_{\bm k \in I} \hat g_{\bm k}\exp(2\pi\mathrm i\langle \bm k, \cdot\rangle)$ the equation $\bm{\hat g} = (\bm L^\ast\bm L)\inv\bm L^\ast\bm f$, where
\begin{equation*}
    \bm L = (\exp(2\pi\mathrm i \langle \bm k, \bm x^i\rangle))_{i=1,\dots, M, \bm k\in I} \,.
\end{equation*}
We will solve this system of equations iterative only using matrix-vector multiplications.
Because of their one-dimensional structure of the \mbox{rank-1} lattices, a one-dimensional FFT can be used to compute the matrix-vector product with the corresponding Fourier matrix $\bm L$
in $\mathcal O(M\log M + d|I|)$ instead of the na\"ive $\mathcal O(M\cdot|I|)$, where $I \subset \mathds Z^d$ is an arbitrary frequency index set.
For approximating functions with \mbox{rank-1} lattices we suppose the following feature:
We say a \mbox{rank-1} lattice $\bm X_M$ has the \emph{reconstructing property} for a frequency index set $I$, if
\begin{equation}\label{eq:recon}
    \frac{1}{M} \sum_{i=1}^{M} \exp(2\pi\mathrm i \langle\bm k, \bm x^i\rangle)
    = \delta_{\bm 0,\bm k}
    \quad\text{for all}\quad
    \bm k\in\mathcal D(I)\,,
\end{equation}
where $\mathcal D(I) = \{\bm k-\bm l : \bm k, \bm l \in I\}$.
Approximation bounds and further resources can be found in \cite{Sloan94, nuyensdiss07, kaemmererdiss, KaPoVo13, PlPoStTa18, DKP22}.

\begin{example}\label{ex:halfrate} When approximating functions from Sobolev spaces with mixed smoothness $H_{\mathrm{mix}}^s$ for $s>1/2$ the best frequency index sets $I$ for approximation are so called are hyperbolic crosses, cf.\ \cite{DTU18}.
    We consider the following two scenarios:
    
    \textbf{1.}
    When approximating with samples from a reconstructing \mbox{rank-1} lattice $\bm X_M$ the following error bound was shown for the least squares approximation in \cite[Theorem~2]{ByKaUlVo16}:
    \begin{equation*}
        M^{-s}
        \lesssim
        \!\sup_{\|f\|_{H_{\mathrm{mix}}^s}\le 1} \|f-S_{\bm X_M} f\|_{L_2}^{2}
        \!\lesssim M^{-s}(\log M)^{(d-2)s+d-1} ,
    \end{equation*}
    where the lower bound holds for all \mbox{rank-1} lattices and there exists a \mbox{rank-1} lattice satisfying the upper one.
     
    \textbf{2.}
    In contrast to that, for the same frequencies from the hyperbolic cross $I$ and using uniformly drawn points $\bm X = \{\bm x^1, \dots, \bm x^n\}$ we obtain by \cite[Corollary~2]{KU21}
    \begin{equation*}
        \sup_{\|f\|_{H_{\text{mix}}^s}\le 1} \|f-S_{\bm X} f\|_{L_2}^{2}
        \lesssim n^{-2s}(\log n)^{2ds} \,.
    \end{equation*}
\end{example} 

Example~\ref{ex:halfrate} demonstrates that the sample complexity loses half the rate of convergence when approximating with \mbox{rank-1} lattices compared to uniformly random points.
The reason for that lies in the reconstructing requirement \eqref{eq:recon} which are $|\mathcal D(I)| \approx |I|^2$ conditions blowing up the size $M$ of the \mbox{rank-1} lattice.
However, when we use the uniformly random points with the better approximation rate, the lack of structure in the uniformly random points prevents the implementation of a fast and efficient matrix-vector multiplication with the corresponding Fourier matrix.

It was show in \cite[Theorem~3.1]{BKPU22} that the good approximation rates and the fast algorithms can be combined:
The approach is to discretely subsample a \mbox{rank-1} lattice to obtain points $\bm X = \{\bm x^1, \dots, \bm x^n\}$ from a \mbox{rank-1} lattice with $n\ge 12|I|(\log |I|+t)$.
Since the underlying structure is preserved fast Fourier algorithms are applicable, cf.\ \cite[Eq.~5.5]{BKPU22}.
Further, we have
\begin{equation}\label{eq:MZ}
    A \|f\|_{L_2}^2
    \le \frac{1}{n} \sum_{i=1}^{n} |f(\bm x^i)|^2
    \le B\|f\|_{L_2}^2
\end{equation}
for all $f\in \spn\{\exp(2\pi\mathrm i\langle\bm k,\cdot\rangle)\}_{\bm k\in I}$ with $A=1/2$, $B = 3/2$, and probability exceeding $1-2\exp(-t)$.
The condition \eqref{eq:MZ} is known as \emph{$L_2$-Marcinkiewicz-Zygmund inequality} and is a relaxation of the reconstructing property \eqref{eq:recon}, since for $A=B=1$ \eqref{eq:MZ} is equivalent to \eqref{eq:recon} which can be shown using the parallelogram law, cf.\ \cite[Theorem~2.3]{BKPU22}.
It gives a relation of the continuous $L_2$-norm and the point evaluations and is used to show error bounds for least squares approximation.
For continuously random points this was done for individual functions in \cite{CDL13, CM17} and improved by \cite{Bar22}.
Note, the existence of a probability density was shown such that \eqref{eq:MZ} holds with merely linear oversampling, cf.\ \cite{DC22}.

The following result is a combination of the discrete subsampling techniques from \cite{BKPU22} and the error bound from \cite[Thm.~3.2]{Bar22} for individual function approximation.

\begin{theorem}\label{thm:err_est} Let $f\colon \mathds T^d\to\mathds C$ be a fixed function and $\bm X_M = \{\bm x^1, \dots, \bm x^M\}\subset\mathds T^d$ be a reconstructing \mbox{rank-1} lattice for a frequency set $I_M\subset \mathds Z^d$.
    Further, let $\emptyset\neq I\subset I_M$, $t>0$, and $n$ be such that $n \ge 12 |I|(\log|I|+t)$.
    Drawing a set $\bm X = \{\bm x^i\}_{i\in J}$, $|J|=n$ of points i.i.d.\ and uniform from $\bm X_M$, we have
    \begin{align*}
        &\| f - S_{\bm X} f \|_{L_2}^{2} \\
        &\quad \le \Big( 3 \|f-P_I f\|_{L_2}
        +\sqrt{\frac{2}{9|I|}} \|P_{I_M}f-P_If\|_{\infty}\Big)^2 \\
        &\qquad + 4\|f-P_{I_M}f\|_{\infty}^2 \\
        &\quad \le \Big( 3+\sqrt{\frac{2|I_M\setminus I|}{9|I|}} \Big)^2 \|f-P_I f\|_{L_2}^2
        + 4\|f-P_{I_M}f\|_{\infty}^2
    \end{align*}
    with probability exceeding $1-2\exp(-t)$.
\end{theorem} 

Given the logarithmic oversampling and assuming $|I_M\setminus I| = c|I|$ for some constant $c>0$,  we obtain the projection error in the first summand, which is the best possible from the given approximation space.
This has to be balanced with the second term which decreases for bigger $I_M$, which is a degree of freedom not affecting the sampling complexity.
In the numerical experiments we will see that $I = I_M$ is sufficient in practice.
Note, that in this case the corresponding rank-1 lattice will still be of size $M\approx|I|^2$ and the random subsampling with logarithmic oversampling will improve the sampling complexity. \section{SFT with subsampled \mbox{rank-1} lattices}\label{sec:sfftsub} 

For a function $f\colon\mathds T^d\to\mathds C$ and a threshold $\delta>0$, the final goal is to find $I_\delta \coloneqq \{\bm k\in\mathds Z^d : |\hat f_{\bm k}| \ge \delta\}$ or a superset of slightly bigger size.
As described in Section~\ref{sec:sft}, our approach works in a dimension-incremental way and so will the analysis.
The goal in the step from dimension $t-1$ to $t$ is the detection of $\mathcal P_{\{1,\dots,t\}} (I_\delta) \subset \mathds Z^{t}$.
We first show that using the projected coefficients \eqref{eq:proj_fc} yields the objective.

\begin{theorem}\label{eq:detectproject} Let $f\colon\mathds T^d\to\mathds C$, $\varepsilon, \delta > 0$, $I_\delta \coloneqq \{\bm k\in\mathds Z^d : |\hat f_{\bm k}| \ge \delta\}$, and
    \begin{equation*}
        r
        \ge 4\Big(|I_\delta| + \frac{1}{\delta^2} \Big( \sum_{\bm k\notin I_\delta}|\hat f_{\bm k}| \Big)^2 \Big)
        \Big(\log|I_\delta| + \log \frac{1}{\varepsilon}\Big)
        \,.
    \end{equation*}
    Further let $\bm\xi^1, \dots, \bm\xi^r \in\mathds T^{d-t}$ be drawn i.i.d.\ uniformly random.
    With probability $1-\varepsilon$ we detect all important frequencies in dimension $t$ via the projected Fourier coefficients \eqref{eq:proj_fc} with $r$ detection iterations and threshold $\delta' \le \delta/\sqrt{2}$, i.e.,
    \begin{equation*}
        \max_{i = 1, \dots, r} | \hat f_{\{1,\dots, t\}, \bm k}(\bm\xi^i) |
        \ge \delta'
        \quad\forall\bm k\in\mathcal P_{\{1,\dots,t\}}(I_\delta) \,.
    \end{equation*}
\end{theorem} 

In practice we do not have the exact projected Fourier coefficients $\hat f_{\{1,\dots,t\},\bm k}(\bm\xi)$.
Rather, we will approximate them by approximating
\begin{equation*}
    f(\cdot, \bm\xi^i)
    = \sum_{\bm k\in\mathds Z^{t}} \hat f_{\{1,\dots,t\},\bm k}(\bm\xi) \exp(2\pi\mathrm i\langle\bm k, \cdot\rangle)
\end{equation*}
for fixed anchors $\bm\xi^1, \dots, \bm\xi^r \in\mathds T^{d-t}$ in the last $d-t$ components and a subsampled \mbox{rank-1} lattice $\bm X\subset\mathds T^t$ in the first $t$ components:
\begin{equation}\label{eq:apprproj}
    S_{\bm X} f(\cdot, \bm\xi)
    = \sum_{\bm k\in\mathds Z^{t}} \hat g_{\{1,\dots,t\},\bm k}(\bm\xi) \exp(2\pi\mathrm i\langle\bm k, \cdot\rangle) \,.
\end{equation}

\begin{theorem}\label{thm:dinc} Let the assumptions from Theorem~\ref{eq:detectproject} hold and let $\mathcal P_{\{1,\dots,t\}}(I_\delta) \subset I_{\{1,\dots,t\}} \subset I_{\{1,\dots,t\}}^M$ be frequency index sets such that $|I_{\{1,\dots,t\}}^M \setminus I_{\{1, \dots, t\}}| \le 9/2 |I_{\{1,\dots, t\}}|$.
    Further, let $\bm X^M$ be a reconstructing \mbox{rank-1} lattice for $I_{\{1,\dots,t\}}^M$ with probability $1-\varepsilon$ and $\bm X\subset\bm X^M$ an i.i.d.\ uniformly drawn subset with 
    \begin{equation*}
        |\bm X| \ge 12 |I_{\{1,\dots,t\}}|
        \Big(
        \log|I_{\{1,\dots,t\}}|
        +\log\Big(\frac{2r}{\varepsilon}\Big)
        \Big) \,.
    \end{equation*}
    With probability $1-3\varepsilon$ we detect all important frequencies in dimension $t$ via the approximated projected Fourier coefficients $\hat g_{\{1,\dots,t\},\bm k}(\bm\xi^i)$ from \eqref{eq:apprproj} with $r$ detection iterations and threshold
    \begin{equation*}
        \delta' \le \frac{\delta}{\sqrt{2}} - 4\|f-P_{I_\delta} f\|_{L_2}
        - 2 \|f - \mathcal P_{I_{\{1,\dots,t\}}^M \times \mathds T^{d-t}} f\|_{\infty} \,,
    \end{equation*}
    i.e.,
    \begin{equation*}
        \max_{i = 1, \dots, r} | \hat g_{\{1,\dots, t\}, \bm k}(\bm\xi^i) |
        \ge \delta'
        \quad\forall\bm k\in\mathcal P_{\{1,\dots,t\}}(I_\delta) \,.
    \end{equation*}
\end{theorem} 

Note, choosing $I_{\{1,\dots,t\}}^M$ large does not affect the sampling complexity but only the initial \mbox{rank-1} lattice from which we sample and diminishes the term $\|f - \mathcal P_{I_{\{1,\dots,t\}}^M \times \mathds T^{d-t}} f\|_{\infty}$.

Having shown the successful detection of the important frequencies in one dimension-incremental step it is left to apply Theorem~\ref{thm:dinc} iteratively to obtain our main theorem stating the successful detection of all important frequencies $\bm k\in I_\delta$ using samples in subsampled \mbox{rank-1} latices.

\begin{theorem}\label{thm:overall} Let $f\colon\mathds T^d\to\mathds C$, $\varepsilon, \delta > 0$, $\Gamma\supset I_\delta \coloneqq \{\bm k\in\mathds Z^d : |\hat f_{\bm k}| \ge \delta\}$, and
    \begin{equation*}
        r
        \ge 4\Big(|I_\delta| + \frac{1}{\delta^2} \Big( \sum_{\bm k\notin I_\delta}|\hat f_{\bm k}| \Big)^2 \Big)
        \Big(\log|I_\delta| + \log \frac{1}{\varepsilon}\Big)
        \,.
    \end{equation*}
    
    \begin{enumerate}
    \item
        Let $t = 1,\dots, t$ and $\bm X_{\{t\}}^M$ be a reconstructing \mbox{rank-1} lattice for $J_{\{t\}} \coloneqq \mathcal P_{\{t\}}(\Gamma)$ with probability $1-\varepsilon$ and $\bm X_{\{t\}}\subset\bm X_{\{t\}}^M$ an i.i.d.\ uniformly drawn subset with 
        \begin{equation*}
            |\bm X_{\{t\}}| \ge 12 |J_{\{t\}}|
            \Big(
            \log|J_{\{t\}}|
            +\log\Big(\frac{2r}{\varepsilon}\Big)
            \Big) \,.
        \end{equation*}
        Further, let $\bm\Xi_{\{t\}} = \{\bm\xi^1, \dots, \bm\xi^r\} \subset\mathds T^{d-1}$ be drawn i.i.d.\ uniformly random.
        
        Using samples in $\bm X_{\{t\}}\times\bm\Xi_{\{t\}}$ for $r$ least squares approximations, we construct $I_{\{t\}}$ such that
        \begin{equation*}
            J_{\{t\}}
            \supset I_{\{t\}}
            \supset \mathcal P_{\{t\}}(I_\delta)
        \end{equation*}
        with probability exceeding $1-3\varepsilon$.
            
    \item
        Let $t = 2,\dots, t$ and $\bm X_{\{1,\dots, t\}}^M$ be a reconstructing \mbox{rank-1} lattice for 
        $J_{\{1,\dots,t\}} \coloneqq (I_{\{1,\dots,t-1\}}\times I_{\{t\}})\cap\mathcal P_{\{1,\dots,t\}}(\Gamma)$
        with probability $1-\varepsilon$ and $\bm X_{\{1,\dots,t\}}\subset\bm X_{\{1,\dots,t\}}^M$ an i.i.d.\ uniformly drawn subset with 
        \begin{equation*}
            |\bm X_{\{1,\dots,t\}}| \ge 12 |J_{\{1,\dots,t\}}|
            \Big(
            \log |J_{\{1,\dots,t\}}|
            +\log\Big(\frac{2r}{\varepsilon}\Big)
            \Big) \,.
        \end{equation*}
        Further, let $\bm\Xi_{\{1,\dots,t\}} = \{\bm\xi^1, \dots, \bm\xi^r\} \subset\mathds T^{d-t}$ be drawn i.i.d.\ uniformly random.
        
        Using samples in $\bm X_{\{1,\dots,t\}}\times\bm\Xi_{\{1,\dots,t\}}$ for $r$ least squares approximations, we construct $I_{\{1,\dots,t\}}$ such that
        \begin{equation*}
            J_{\{1,\dots,t\}}
            \supset I_{\{1,\dots,t\}}
            \supset \mathcal P_{\{1,\dots,t\}}(I_\delta)
        \end{equation*}
        with probability exceeding $1-3\varepsilon$.
    \end{enumerate}
    In particular, we have $I_{\{1,\dots, d\}} \supset I_\delta$ with probability exceeding $1-6d\varepsilon$.
\end{theorem} 

\begin{proof} The assertion follows from repeatedly applying Theorem~\ref{thm:dinc} and union bound.
\end{proof} %
 \section{Numerical experiments}\label{sec:numerics} 

We consider the $10$-dimensional test function $f\colon \mathds T^{10} \rightarrow \mathds R,$
\begin{equation*}
    f(\bm x) \coloneqq\!\!\! \prod_{t\in\{ 1,3,8 \}}\!\!\! N_2(x_t) +\!\!\! \prod_{t\in\{ 2,5,6,10 \}}\!\!\! N_4(x_t) +\!\!\! \prod_{t\in\{ 4,7,9 \}}\!\!\! N_6(x_t),
\end{equation*}
with $N_m\colon \mathds T\rightarrow\mathds R$ being the B-Spline of order $m\in\mathds N$
\begin{equation*}
    N_m(x)
    \coloneqq C_m \sum_{k\in\mathds Z} (-1)^k \sinc\Big(\frac{\pi k}{m} \Big)^m \exp(2\pi\mathrm i k x) \,,
\end{equation*}
with a constant $C_m>0$ such that $\| N_m \|_{L_2(\mathds T)} = 1$.
This function was already used to test high-dimensional algorithms in \cite{volkmerdiss, PV16, KPV21, KKV22, KPT22}.
It is not a sparse signal with respect to the trigonometric system, i.e., $|\hat f_{\bm k}| \neq 0$ for infinitely many $\bm k\in\mathds Z^{10}$.

Because of the product structure and the smoothness of the B-splines, we have $|\hat f_{\bm k}| \le C\prod_{t=1}^{d} \max\{1, k_t\}^{-2}$ for some constant $C>0$.
For this reason we choose our initial search space $\Gamma\subset\mathds Z^{10}$ to be a $10$-dimensional hyperbolic cross of radius $2^8$, i.e.,
\begin{equation*}
    \Gamma
    = \Big\{ \bm k \in \mathds Z^d: \prod_{t=1}^{10} \max \{ 1, |k_t| \} \leq 2^8 \Big\}
\end{equation*}
with $|\Gamma| = 8\,827\,703\,433$ possible frequencies to choose from.

In our algorithm we choose the number of detection iterations $r=5$ and the detection threshold to $\delta' = 10^{-12}$.
To limit the number of detected frequencies, we set the target sparsity to $s\in\{2^3, \dots, 2^{13}\}$ and only consider the $s_\mathrm{local} = \lceil 1.2 s\rceil$ largest approximated Fourier coefficients in each dimension-incremental step.
We then used three different sampling strategies for the dimension-incremental steps:
\begin{itemize}
    \item full \mbox{rank-1} lattices utilizing fast Fourier algorithms as in \cite{PV16};
    \item i.i.d.\ uniformly random points utilizing good sampling complexity as in \cite{KPT22};
    \item subsampled \mbox{rank-1} lattices combining both advantages
\end{itemize}
and obtained a frequency set $I\subset\mathds Z^{10}$ and Fourier coefficients $\bm{\hat g} = (\hat g_{\bm k})_{\bm k\in I}$ of our approximation.
All tests were performed $10$ times in MATLAB\textsuperscript{\textregistered} using $2$ six core CPUs Intel\textsuperscript{\textregistered} Xeon\textsuperscript{\textregistered} CPU E5-2620 v3 @ 2.40GHz and 64 GB RAM.
We stopped computations which exeeded a time limit of $1$ hour.

\begin{figure} \centering
    \begingroup
  \makeatletter
  \providecommand\color[2][]{\GenericError{(gnuplot) \space\space\space\@spaces}{Package color not loaded in conjunction with
      terminal option `colourtext'}{See the gnuplot documentation for explanation.}{Either use 'blacktext' in gnuplot or load the package
      color.sty in LaTeX.}\renewcommand\color[2][]{}}\providecommand\includegraphics[2][]{\GenericError{(gnuplot) \space\space\space\@spaces}{Package graphicx or graphics not loaded}{See the gnuplot documentation for explanation.}{The gnuplot epslatex terminal needs graphicx.sty or graphics.sty.}\renewcommand\includegraphics[2][]{}}\providecommand\rotatebox[2]{#2}\@ifundefined{ifGPcolor}{\newif\ifGPcolor
    \GPcolortrue
  }{}\@ifundefined{ifGPblacktext}{\newif\ifGPblacktext
    \GPblacktexttrue
  }{}\let\gplgaddtomacro\g@addto@macro
\gdef\gplbacktext{}\gdef\gplfronttext{}\makeatother
  \ifGPblacktext
\def\colorrgb#1{}\def\colorgray#1{}\else
\ifGPcolor
      \def\colorrgb#1{\color[rgb]{#1}}\def\colorgray#1{\color[gray]{#1}}\expandafter\def\csname LTw\endcsname{\color{white}}\expandafter\def\csname LTb\endcsname{\color{black}}\expandafter\def\csname LTa\endcsname{\color{black}}\expandafter\def\csname LT0\endcsname{\color[rgb]{1,0,0}}\expandafter\def\csname LT1\endcsname{\color[rgb]{0,1,0}}\expandafter\def\csname LT2\endcsname{\color[rgb]{0,0,1}}\expandafter\def\csname LT3\endcsname{\color[rgb]{1,0,1}}\expandafter\def\csname LT4\endcsname{\color[rgb]{0,1,1}}\expandafter\def\csname LT5\endcsname{\color[rgb]{1,1,0}}\expandafter\def\csname LT6\endcsname{\color[rgb]{0,0,0}}\expandafter\def\csname LT7\endcsname{\color[rgb]{1,0.3,0}}\expandafter\def\csname LT8\endcsname{\color[rgb]{0.5,0.5,0.5}}\else
\def\colorrgb#1{\color{black}}\def\colorgray#1{\color[gray]{#1}}\expandafter\def\csname LTw\endcsname{\color{white}}\expandafter\def\csname LTb\endcsname{\color{black}}\expandafter\def\csname LTa\endcsname{\color{black}}\expandafter\def\csname LT0\endcsname{\color{black}}\expandafter\def\csname LT1\endcsname{\color{black}}\expandafter\def\csname LT2\endcsname{\color{black}}\expandafter\def\csname LT3\endcsname{\color{black}}\expandafter\def\csname LT4\endcsname{\color{black}}\expandafter\def\csname LT5\endcsname{\color{black}}\expandafter\def\csname LT6\endcsname{\color{black}}\expandafter\def\csname LT7\endcsname{\color{black}}\expandafter\def\csname LT8\endcsname{\color{black}}\fi
  \fi
    \setlength{\unitlength}{0.0500bp}\ifx\gptboxheight\undefined \newlength{\gptboxheight}\newlength{\gptboxwidth}\newsavebox{\gptboxtext}\fi \setlength{\fboxrule}{0.5pt}\setlength{\fboxsep}{1pt}\definecolor{tbcol}{rgb}{1,1,1}\begin{picture}(4700.00,4120.00)\gplgaddtomacro\gplbacktext{\csname LTb\endcsname \put(560,2562){\makebox(0,0)[r]{\strut{}$10^{-4}$}}\csname LTb\endcsname \put(560,3177){\makebox(0,0)[r]{\strut{}$10^{-2}$}}\csname LTb\endcsname \put(560,3792){\makebox(0,0)[r]{\strut{}$10^{0}$}}\csname LTb\endcsname \put(851,2357){\makebox(0,0){\strut{}$2^{4}$}}\csname LTb\endcsname \put(1454,2357){\makebox(0,0){\strut{}$2^{8}$}}\csname LTb\endcsname \put(2056,2357){\makebox(0,0){\strut{}$2^{12}$}}\csname LTb\endcsname \put(1450,3977){\makebox(0,0){\strut{}realative $L_2$-error}}\csname LTb\endcsname \put(1450,2192){\makebox(0,0){\strut{}$s$}}}\gplgaddtomacro\gplfronttext{}\gplgaddtomacro\gplbacktext{\csname LTb\endcsname \put(2900,2562){\makebox(0,0)[r]{\strut{}$10^{-7}$}}\csname LTb\endcsname \put(2900,3089){\makebox(0,0)[r]{\strut{}$10^{-4}$}}\csname LTb\endcsname \put(2900,3616){\makebox(0,0)[r]{\strut{}$10^{-1}$}}\csname LTb\endcsname \put(3191,2357){\makebox(0,0){\strut{}$2^{4}$}}\csname LTb\endcsname \put(3794,2357){\makebox(0,0){\strut{}$2^{8}$}}\csname LTb\endcsname \put(4396,2357){\makebox(0,0){\strut{}$2^{12}$}}\csname LTb\endcsname \put(3790,3977){\makebox(0,0){\strut{}$\|\bm{\hat f}-\bm{\hat g}\|_{\infty}$}}\csname LTb\endcsname \put(3790,2192){\makebox(0,0){\strut{}$s$}}}\gplgaddtomacro\gplfronttext{}\gplgaddtomacro\gplbacktext{\csname LTb\endcsname \put(560,512){\makebox(0,0)[r]{\strut{}$10^{4}$}}\csname LTb\endcsname \put(560,1127){\makebox(0,0)[r]{\strut{}$10^{7}$}}\csname LTb\endcsname \put(560,1742){\makebox(0,0)[r]{\strut{}$10^{10}$}}\csname LTb\endcsname \put(851,307){\makebox(0,0){\strut{}$2^{4}$}}\csname LTb\endcsname \put(1454,307){\makebox(0,0){\strut{}$2^{8}$}}\csname LTb\endcsname \put(2056,307){\makebox(0,0){\strut{}$2^{12}$}}\csname LTb\endcsname \put(1450,1927){\makebox(0,0){\strut{}number of samples}}\csname LTb\endcsname \put(1450,142){\makebox(0,0){\strut{}$s$}}}\gplgaddtomacro\gplfronttext{}\gplgaddtomacro\gplbacktext{\csname LTb\endcsname \put(2900,512){\makebox(0,0)[r]{\strut{}$10^{0}$}}\csname LTb\endcsname \put(2900,1127){\makebox(0,0)[r]{\strut{}$10^{2}$}}\csname LTb\endcsname \put(2900,1742){\makebox(0,0)[r]{\strut{}$10^{4}$}}\csname LTb\endcsname \put(3191,307){\makebox(0,0){\strut{}$2^{4}$}}\csname LTb\endcsname \put(3794,307){\makebox(0,0){\strut{}$2^{8}$}}\csname LTb\endcsname \put(4396,307){\makebox(0,0){\strut{}$2^{12}$}}\csname LTb\endcsname \put(3790,1927){\makebox(0,0){\strut{}computation time (s)}}\csname LTb\endcsname \put(3790,142){\makebox(0,0){\strut{}$s$}}}\gplgaddtomacro\gplfronttext{}\gplbacktext
    \put(0,0){\includegraphics[width={235.00bp},height={206.00bp}]{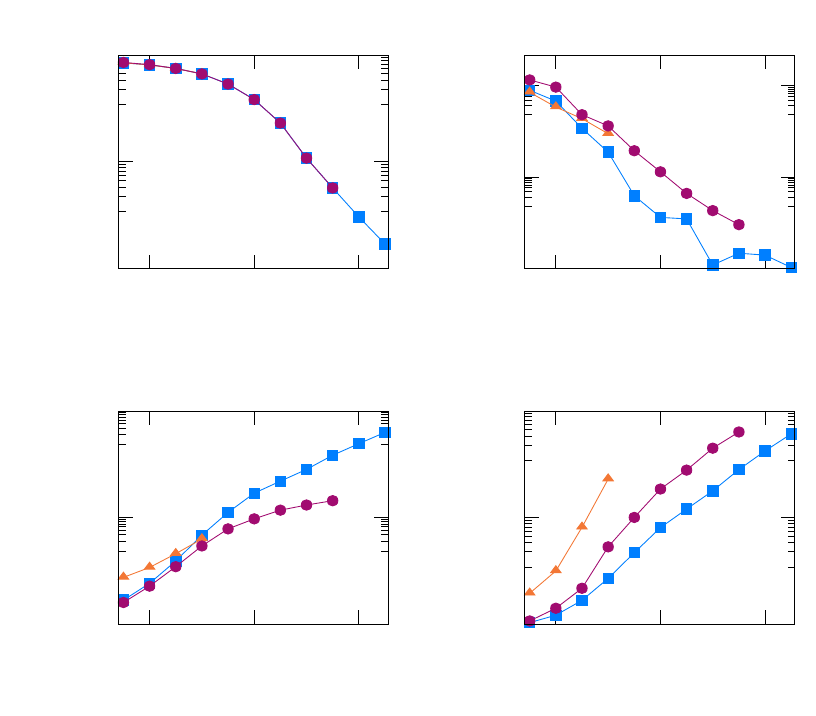}}\gplfronttext
  \end{picture}\endgroup
     \caption{
        Approximation results for the $10$-dimensional test function for
        \textcolor{azure}{azure:} full \mbox{rank-1} lattices,
        \textcolor{orange}{orange:} i.i.d.\ uniformly random points, and
        \textcolor{violet}{violet:} subsampled \mbox{rank-1} lattices.
    }\label{fig:test10d}
\end{figure} 

In Figure~\ref{fig:test10d} we depicted the medians of various quantities of the experiment.
The two upper graphs show that the relative $L_2$-error as well as the error in the coefficients behaves the same with all three approaches.
Note, that we did not need as many detection iterations as we proposed in the theory but $r=5$ was sufficient.
Next, we investigate the sampling complexity and computation time.

\textbf{Sampling complexity.}
As discussed in the theoretical part, the reconstructing requirement of the \mbox{rank-1} lattice blows up its size resulting in the most used sampling points.
Because of computational infeasibility, we do not have many computations with the i.i.d.\ uniformly random points and cannot capture its behaviour (in our experience it should behave similar to the subsampled \mbox{rank-1} lattice).
The subsampled \mbox{rank-1} lattices have better sampling complexity than the full \mbox{rank-1} lattices and the graph suggests this advantage will increase for higher sparsity $s$.

\textbf{Computation time.}
The fastest computation time can be seen with the full \mbox{rank-1} lattices since the approximations only need one matrix-vector product each for which fast Fourier algorithms are utilized.
The subsampled \mbox{rank-1} lattices are slower by a constant factor of $10$.
Here, the same fast Fourier algorithm is used but the approximations use an iterative solver.
We capped the maximal number of iterations by $10$, which explains the constant factor.
For the i.i.d.\ uniformly random points no fast Fourier algorithms are available making it the slowest approach.

The experiments confirms our theoretical findings of subsampled \mbox{rank-1} lattices combining the computational and sampling advantages of full \mbox{rank-1} lattices and random points, respectively.

\bibliographystyle{amsplain}
\providecommand{\bysame}{\leavevmode\hbox to3em{\hrulefill}\thinspace}
\providecommand{\MR}{\relax\ifhmode\unskip\space\fi MR }
\providecommand{\MRhref}[2]{\href{http://www.ams.org/mathscinet-getitem?mr=#1}{#2}
}
\providecommand{\href}[2]{#2}

\newpage
\onecolumn
\appendix
\section{Proofs} 

\begin{proof}[Proof of Theorem~II.2] The proof follows ideas from \cite[Theorem~3.20]{LPU21} or \cite[Theorem~3.2]{Bar22} and utilizes the discrete subsampling from \cite{BKPU22}.
    By \cite[Lemma~2.2 and Theorem~3.1]{BKPU22} we obtain
    \begin{equation}\label{eq:ldskjf}
        \frac{n}{2}
        \le \sigma_{\min}^2 (\bm L)
        = \|(\bm L^\ast\bm L)\inv\bm L^\ast\|_{2\to 2}^2
    \end{equation}
    with probability $1-\exp(-t)$.
    We split the approximation error as follows
    \begin{align*}
        \| f - S_{\bm X} \bm f \|_{L_2}^{2}
        &= \|f-P_I f\|_{L_2}^2
        + \| P_I f - S_{\bm X} \bm f \|_{ L_2}^{2} \\
        &\le \|f-P_I f\|_{L_2}^2
        + 2\| P_I f - S_{\bm X} P_{I_M} \bm f \|_{ L_2}^{2}
        + 2\| S_{\bm X} P_{I_M} f - S_{\bm X} \bm f \|_{ L_2}^{2} \,.
    \end{align*}
    By event \eqref{eq:ldskjf} and the invariance of $S_{\bm X}$ to functions supported on $I$, we obtain
    \begin{align}
        \| f - S_{\bm X} \bm f \|_{L_2}^{2}
        &\le \|f-P_I f\|_{L_2}^2
        + \frac{4}{n} \sum_{i\in J} |(P_I f - P_{I_M} f)(\bm x^i)|^2
        + \frac{4}{n} \sum_{i\in J} |(f - P_{I_M} f)(\bm x^i)|^2 \nonumber \\
        &\le \|f-P_I f\|_{L_2}^2
        + \frac{4}{n} \sum_{i\in J} |(P_I f - P_{I_M} f)(\bm x^i)|^2
        + 4 \|f - P_{I_M} f\|_\infty^2 \,. \label{eq:lkjfs}
    \end{align}
    
    We will estimate the middle summand by Bernstein's inequality, cf.\ \cite[Corollary 7.31]{FR13} or \cite[Theorem 6.12]{SC08}.
    We define random variables $\xi_i = |(P_I f-P_{I_M} f)(\bm x^i)|^2 - \|P_{I_M}f-P_If\|_{L_2}^2$.
    By the reconstructing property, we have $\frac{1}{M} \sum_{j=1}^{M} |(P_I f-P_{I_M}f)(\bm x^i)|^2 = \|P_{I_M}f-P_If\|_{L_2}^2$ and, thus, $\xi_i$ are mean zero.
    Further, we have
    \begin{align*}
        \|\xi_i\|_\infty
        \le (\|P_If-P_{I_M} f\|_{\infty} + \|P_If-P_{I_M} f\|_{L_2})^2
    \end{align*}
    and
    \begin{align*}
        \mathds E( \xi_i^2 )
        &= \mathds E\Big( \Big( |(P_If-P_{I_M}f)(\bm x^i)|^2 \Big)^2\Big) + \|P_If-P_{I_M} f\|_{L_2}^4 \\
        &\le \|P_If-P_{I_M} f\|_{\infty}^2 \sum_{i=1}^{M} |(P_If-P_{I_M}f)(\bm x^i)|^2 + \|P_If-P_{I_M} f\|_{L_2}^4 \\
        &\le \|P_If-P_{I_M} f\|_{L_2}^2 (\|P_If-P_{I_M} f\|_{\infty} + \|P_If-P_{I_M} f\|_{L_2} )^2 \,.
    \end{align*}
    Applying Bernstein's inequality yields
    \begin{align*}
        &\frac{1}{n} \sum_{i\in J}|(P_If-P_{I_M}f)(\bm x^i)|^2 - \|P_If-P_{I_M}f\|_{L_2}^2
        \le \frac{2t}{3n} (\|P_If-P_{I_M} f\|_{\infty} + \|P_If-P_{I_M} f\|_{L_2})^2 \\
        &\quad + \sqrt{\frac{2t}{n}} \|P_If-P_{I_M} f\|_{L_2} (\|P_If-P_{I_M} f\|_{\infty} + \|P_If-P_{I_M} f\|_{L_2} )
    \end{align*}
    with probability $1-\exp(-t)$.
    
    Plugging this in \eqref{eq:lkjfs} and using $\|P_{I_M}f-P_If\|_{L_2}\le\|f-P_If\|_{L_2}$, we obtain
    \begin{align*}
        \| f - S_{\bm X} \bm f \|_{L_2}^{2}
        &\le 5 \|f-P_I f\|_{L_2}^2
        + 2\sqrt{\frac{8t}{n}} \|f-P_If\|_{L_2}\Big(\|P_{I_M}f-P_If\|_{\infty} + \|f-P_If\|_{L_2}\Big) \\
        &\quad + \frac{8t}{3n} \Big(\|P_{I_M}f-P_If\|_{\infty} + \|f-P_If\|_{L_2}\Big)^2
        + 4\|f-P_{I_M}f\|_{\infty}^2 \\
        &\le \Big( \sqrt{5} \|f-P_I f\|_{L_2}
        +\sqrt{\frac{8t}{3n}} \Big(\|P_{I_M}f-P_If\|_{\infty} + \|f-P_If\|_{L_2}\Big) \Big)^2 \\
        &\quad+ 4\|f-P_{I_M}f\|_{\infty}^2 \,.
    \end{align*}
    Using the assumption on $n$, we have $8t/3n \le 2/9|I| \le 2/9$ and further estimate
    \begin{align*}
        \| f - S_{\bm X} \bm f \|_{L_2}^{2}
        &\le \Big( 3 \|f-P_I f\|_{L_2}
        +\sqrt{\frac{2}{9|I|}} \|P_{I_M}f-P_If\|_{\infty}\Big)^2
        + 4\|f-P_{I_M}f\|_{\infty}^2 \,.
    \end{align*}
    The second part of the bound is obtained by applying Hölder's inequality:
    \begin{align*}
        \|P_{I_M}f-P_If\|_{\infty}^2
        &= \Big\| \sum_{\bm k\in I_M\setminus I} \hat f_{\bm k} \exp(2\pi\mathrm i\langle\bm k,\cdot\rangle) \Big\|_{\infty}^2 \\
        &\le \Big\| \sum_{\bm k\in I_M\setminus I} |\exp(2\pi\mathrm i\langle\bm k,\cdot\rangle)|^2 \Big\|_{\infty}
        \sum_{\bm k\in I_M\setminus I} |\hat f_{\bm k}|^2 \\
        &= |I_M\setminus I|\|P_{I_M}f-P_If\|_{L_2}^2
        \le |I_M\setminus I|\|f-P_If\|_{L_2}^2 \,.
    \end{align*}
    By union bound we obtain the overall probability.
\end{proof} 

\begin{lemma}\label{squared} For $X$ uniformly distributed on $\mathds T^d$, we have
    \begin{equation*}
        \mathds P( |g(X)| \ge \delta')
        \ge \frac{\|g\|_{L_2}^2-\delta'^2}{\|g\|_{\infty}^2} \,.
    \end{equation*}
\end{lemma} 

\begin{proof} Consequence of \cite[Par.~9.3.A]{Love77} for $h(t) = t^2$.
\end{proof} 

\begin{proof}[Proof of Theorem~III.1] From Lemma~\ref{squared} follows
    \begin{equation*}
        \mathds P ( |\hat f_{\{1,\dots, t\}, \bm k}(\bm\xi) | \ge \delta' )
        \ge \frac{\|\hat f_{\{1,\dots, t\}, \bm k}\|_{L_2}^2 - \delta'^2}{\|\hat f_{\{1,\dots, t\}, \bm k}\|_{\infty}^2}
        \eqqcolon p_{\bm k} \,.
    \end{equation*}
    Let $\overline p = \max_{\bm k\in\mathcal P_{\{1,\dots, t\}}(I_\delta)} p_{\bm k}$.
    By union bound, we have the successful detection with probability exceeding $1-|I_\delta|(1-\overline p)^r$.
    This is smaller or equal $1-\varepsilon$ whenever
    \begin{equation}\label{eq:asdpojoh}
        r
        \le -\frac{\log |I_\delta| + \log(1/\varepsilon)}{\log(1-\overline p)} \,.
    \end{equation}
    Since $p_{\bm k}\in(0,1)$, we obtain
    \begin{equation*}
        -\frac{1}{\log(1-p_{\bm k})}
        \le \frac{1}{p_{\bm k}}
        = \frac{\|\hat f_{\{1,\dots, t\}, \bm k}\|_{\infty}^{2}}{\|\hat f_{\{1,\dots, t\}, \bm k}\|_{L_2}^{2} - \delta'^2}
        \le \frac{
        ( \sum_{(\bm k, \bm l)\in\mathds Z^{d-t}} |\hat f_{(\bm k, \bm l)} | )^2
        }{
        \sum_{(\bm k, \bm l)\in\mathds Z^{d-t}} |\hat f_{(\bm k, \bm l)} |^2
        - \delta'^2} \,.
    \end{equation*}
    With $\bm k\in\mathcal P_{\{1,\dots,t\}}(I_\delta)$ we have
    \begin{equation*}
        \delta'^2
        \le \delta^2/2
        \le \frac{1}{2} \sum_{(\bm k, \bm l)\in\mathds Z^{d-t}} |\hat f_{(\bm k, \bm l)} |^2 \,.
    \end{equation*}
    Consequently,
    \begin{align*}
        &-\frac{1}{\log(1-p_{\bm k})}
        \le 2 \frac{
        ( \sum_{(\bm k, \bm l)\in\mathds Z^{d-t}} |\hat f_{(\bm k, \bm l)} | )^2
        }{
        \sum_{(\bm k, \bm l)\in\mathds Z^{d-t}} |\hat f_{(\bm k, \bm l)} |^2
        } \\
        &\quad \le 4 \frac{
        ( \sum_{(\bm k, \bm l)\in\mathcal P_{\{1, \dots, t\}, \bm k}(I_\delta) } |\hat f_{(\bm k, \bm l)} | )^2
        }{
        \sum_{(\bm k, \bm l)\in\mathds Z^{d-t}} |\hat f_{(\bm k, \bm l)} |^2
        }
        + 4 \frac{
        ( \sum_{(\bm k, \bm l)\notin\mathcal P_{\{1, \dots, t\}, \bm k}(I_\delta) } |\hat f_{(\bm k, \bm l)} | )^2
        }{
        \sum_{(\bm k, \bm l)\in\mathds Z^{d-t}} |\hat f_{(\bm k, \bm l)} |^2
        } \\
        &\quad \le 4 |\mathcal P_{\{1,\dots,t\}, \bm k}(I_\delta)|
        + \frac{4}{\delta^2}
        \Big( \sum_{(\bm k, \bm l)\notin\mathcal P_{\{1, \dots, t\}, \bm k}(I_\delta) } |\hat f_{(\bm k, \bm l)} | \Big)^2 \\
        &\quad \le 4 |I_{\delta}|
        + \frac{4}{\delta^2} \Big( \sum_{\bm l\notin I_\delta } |\hat f_{\bm l} | \Big)^2 \,.
    \end{align*}
    Plugging this into \eqref{eq:asdpojoh} we obtain the assertion.
\end{proof} 

\begin{proof}[Proof of Theorem~III.2] By Theorem~III.1 we have
    \begin{align*}
        \frac{\delta}{\sqrt{2}}
        &\le \max_{i=1,\dots, r} |\hat f_{\{1,\dots, t\},\bm k}(\bm\xi^i) | \\
        &\le \max_{i=1,\dots, r} |\hat f_{\{1,\dots, t\},\bm k}(\bm\xi^i)
        - \hat g_{\{1,\dots, t\},\bm k}(\bm\xi^i) |
        + |\hat g_{\{1,\dots, t\},\bm k}(\bm\xi^i) | \,.
    \end{align*}
    Next, we show a bound on the above difference in order to obtain the desired threshold.
    We have
    \begin{align*}
        |\hat f_{\{1,\dots, t\},\bm k}(\bm\xi^i) - \hat g_{\{1,\dots, t\},\bm k}(\bm\xi^i) |^2
        &\le \sum_{k\in I_{\{1,\dots, t\}}} | \hat f_{\{1,\dots, t\},\bm k}(\bm\xi^i) - \hat g_{\{1,\dots, t\},\bm k}(\bm\xi^i) |^2 \\
        &\le \| f(\cdot, \bm\xi^i) - S_{\bm X}f(\cdot,\bm\xi^i\|_{L_2}^2 \,.
    \end{align*}
    Using Theorem~2.2 and the assumption on the carnality of the frequency index sets, we obtain
    \begin{align*}
        &|\hat f_{\{1,\dots, t\},\bm k}(\bm\xi^i) - \hat g_{\{1,\dots, t\},\bm k}(\bm\xi^i) |^2 \\
        &\quad\le \Big(3+\sqrt{\frac{2|I_{\{1,\dots,t\}}^{M} \setminus I_{\{1,\dots,t\}}|}{9|I_{\{1,\dots,t\}}|}} \Big)^2
        \|f(\cdot, \bm\xi^i) - P_{I_{\{1,\dots,t\}}} f(\cdot, \bm\xi^i) \|_{L_2}^2 \\
        &\qquad + 4 \|f(\cdot, \bm\xi^i) - P_{I_{\{1,\dots,t\}}^M} f(\cdot, \bm\xi^i) \|_{\infty}^2 \\
        &\quad\le 16 \|f(\cdot, \bm\xi^i) - P_{I_{\{1,\dots,t\}}} f(\cdot, \bm\xi^i) \|_{L_2}^2
        + 4 \|f(\cdot, \bm\xi^i) - P_{I_{\{1,\dots,t\}}^M} f(\cdot, \bm\xi^i) \|_{\infty}^2
    \end{align*}
    with probability
    \begin{equation*}
        1-2r\exp(-t)
        \quad\text{and}\quad
        |\bm X| \ge 12 |I_{\{1,\dots,t\}}|(\log|I_{\{1,\dots,t\}}|+t) \,,
    \end{equation*}
    where union bound was used over the detection iterations.
    Setting $t=\log(2r/\varepsilon)$ we obtain the assertion.
\end{proof}  
\end{document}